\documentclass[11pt,a4paper,twoside]{amsart}

\usepackage{amsmath,amsfonts,amsthm,amsopn,color,amssymb,enumitem,cite,soul}
\usepackage{palatino}
\usepackage{graphicx}
\usepackage[normalem]{ulem}
\usepackage[colorlinks=true,urlcolor=blue,citecolor=red,linkcolor=blue,linktocpage,pdfpagelabels,bookmarksnumbered,bookmarksopen]{hyperref}
\hypersetup{urlcolor=blue, citecolor=red, linkcolor=blue}
\usepackage[left=3.4cm,right=3.4cm,top=2.72cm,bottom=2.72cm]{geometry}

\usepackage[hyperpageref]{backref}

\usepackage[colorinlistoftodos]{todonotes}
\makeatletter
\providecommand\@dotsep{5}
\def\listtodoname{List of Todos}
\def\listoftodos{\@starttoc{tdo}\listtodoname}
\makeatother

\newcommand{\eps}{\varepsilon}

\newcommand{\R}{\mathbb{R}}

\newcommand{\RD}{{\mathbb{R}^2}}
\newcommand{\RT}{{\mathbb{R}^3}}

\renewcommand{\le}{\leslant}
\renewcommand{\ge}{\geslant}
\renewcommand{\a }{\alpha }

\renewcommand{\b }{\beta }

\renewcommand{\d }{\delta }

\newcommand{\g }{\gamma }

\renewcommand{\l }{\lambda}

\newcommand{\n }{\nabla }
\newcommand{\s }{\sigma }

\newcommand{\N}{\mathbb{N}}

\newcommand{\ird }{\int_{\RD}}

\def\bbm[#1]{\mbox{\boldmath $#1$}}
\newcommand{\beq }{\begin{equation}}
\newcommand{\eeq }{\end{equation}}

\renewcommand{\le}{\leqslant}
\renewcommand{\ge}{\geqslant}

\newtheorem{theorem}{Theorem}[section]
\newtheorem{lemma}[theorem]{Lemma}

\newtheorem{proposition}[theorem]{Proposition}
\newtheorem{remark}[theorem]{Remark}

\title[The 2-dimensional nonlinear Schr\"odinger-Maxwell system]{The 2-dimensional nonlinear Schr\"odinger-Maxwell system}

\author[A. Azzollini]{Antonio Azzollini}
\address{Dipartimento di Matematica, Informatica ed Economia, Universit\`a degli
	Studi della Basilicata,
	\newline\indent
	Via dell'Ateneo Lucano 10, I-85100
	Potenza, Italy}
\email{antonio.azzollini@unibas.it}

\author[M. Pimenta]{Marcos T.O. Pimenta}
\address{Departamento de Matem\'atica e Computa\c{c}\~ao, Universidade Estadual Paulista - Unesp, 
	\newline\indent
	CEP: 19060-900, Presidente Prudente - SP, Brazil}
\email{marcos.pimenta@unesp.br}

\subjclass[2010]{35J50, 35Q40}
\keywords{Schr\"odinger-Maxwell system; logarithmic convolution potential; elliptic systems}

\begin{document}
	
		\begin{abstract}
In this paper we carry on the study of a system recently introduced by the first author as the planar version of the well known electrostatic Schr\"odinger - Maxwell equations.

In the positive potential case, we exhibit situations where the existence of solutions depends on the strength of the coupling, being this one modulated by a parameter.

We also present some results in the case of a sign-changing potential.
		
	\end{abstract}

	\maketitle
	\section{Introduction}
		In this paper we wish to perform a deeper study of the following Schr\"odinger-Maxwell problem 
		\beq \label{eq:e01}\tag{$\mathcal{P}_0$}
		\left	\{
		\begin{array}{l}
			-\Delta u  +V(x) u- q \phi u +W'(u)=0\hbox{ in } \R^2,	\\
			\Delta \phi = u^2 \hbox{ in } \R^2,
		\end{array}
		\right.
		\eeq
	introduced in \cite{A}.\\
	The origin of this system is purely physical. Indeed, the system \eqref{eq:e01} is obtained by coupling the nonlinear Schr\"odinger equation with the Maxwell equations and it represents a model for describing the dynamics of a charged particle interacting with the electromagnetic field generated by itself. 
	
	Due to its important meaning, this system was widely studied in three-dimensions by many authors, since it was introduced, in its linear version, in the pioneering paper by Benci and Fortunato \cite{BF} as an eigenvalue problem in a bounded domain. In the last 25 years the literature has been enriched by a lot of  papers dealing with several situations. In particular, we recall fundamental contributions from \cite{ARu,C,DM,K,Ruiz} in looking for existence and multiplicity results, specifically concerning radial solutions, for the problem
		\beq \label{eq:sm}
\left	\{
\begin{array}{l}
	-\Delta u  + u- q \phi u -|u|^{p-2}u=0\hbox{ in } \R^3,	\\
	\Delta \phi = u^2 \hbox{ in } \R^3.
\end{array}
\right.
\eeq

Abandoning the radial function constraint, in \cite{AP} it was proved the existence of a ground state solution for \eqref{eq:sm} when $p\in (3,6)$, and for a nonautonomous version of  \eqref{eq:sm} obtained introducing a Rabinowitz type function $V(x)$ as a coefficient of  the linear term $u$ when $p\in (4,6)$ (the range $p\in (3,4]$ was later successfully treated in \cite{ZZ}). Ranges of nonexistence depending on values of $p$ and $q$ are provided in \cite{AP,DM2,Ruiz} . We also recall \cite{ADP} where system \eqref{eq:sm} was considered in presence of a general nonlinearity and \cite{Ruiz2,IR}  where it was studied an equation which is equivalent to a zero mass version of \eqref{eq:sm}.  

In spite of the broad interest reserved to problem \eqref{eq:sm}, up to our knowledge the literature on \eqref{eq:e01} is definitely less rich. 	
	
	  Motivated by physical considerations on positiveness of energy density related with the rest mass (see \cite{M} for more details), the problem \eqref{eq:e01} was introduced in \cite{A} for $V(x)=1+|x|^\a$  and $W(s)\ge 0$.  In particular, in \cite{A} it was studied the problem
	  
	  		\beq \label{eq:e1}\tag{$\mathcal{P}_1$}
	  \left	\{
	  \begin{array}{l}
	  	-\Delta u  + (1+|x|^\a) u- q \phi u +|u|^{p-2}u=0\hbox{ in } \R^2,	\\
	  	\Delta \phi = u^2 \hbox{ in } \R^2,
	  \end{array}
	  \right.
	  \eeq
	  and proved the following result
		\begin{theorem}\label{A}
			Assume $2<p$. If it holds one or the other of these situations
				\begin{itemize}
					\item $p>4$ and $\a \in \left(0,\tilde p \right]$ with  $\tilde p=\frac{2p-4}{p-4}$,
					\item $2<p\le 4$ and  $\a >0$,
				\end{itemize}   
			then for any $q>0$ \eqref{eq:e1} possesses infinitely many nontrivial  solutions.
		\end{theorem}
	Even if  such a multiplicity result was obtained for $q=1$, the proof  can be repeated without any variation for an arbitrary $q>0$. 
	
	Looking at the proof of  Theorem \ref{A}, we observe how strongly the choice of the exponent $\a$ influences geometrical and compactness properties of the functional associated to the problem for $p>4$. At a first sight,  difficulties in applying usual variational minimax arguments for  $p>4$ and large values of $\a$ could appear exquisitely technical, but it is not the case.\\
	Indeed, as a first result, in Theorem \ref{th:nonex} we use a new estimate to show that the existence of  solutions to \eqref{eq:e1} is, in fact, compromised when the impact of the coupling term $\phi u$ is, roughly speaking, not strong enough.

	In this sense, the main purpose of this paper is to study existence of solutions to \eqref{eq:e1} assuming $p>4$, $\a>\tilde p$, and letting $q$ play the role of a parameter, or, when needed, of a variable of the problem, in order to modulate the effect of coupling. 
	
	\vskip 0,2cm
	An interesting feature coming out of our study is that, under suitable assumptions, we can find both radial and nonradial solutions of \eqref{eq:e1}. 
	
	\vskip 0,2cm
	
	As to the radial case, there are a lot of surprising analogies  between \eqref{eq:e1} under conditions $p>4, \a >\tilde p$ and the three-dimensional Schr\"odinger-Maxwell system \eqref{eq:sm} studied by Ruiz \cite{Ruiz} under assumption $p\in (2,3]$.\\
	In particular, we use a new estimate (see Lemma \ref{le:ineq}) and exploit a variant of Strauss radial lemma to prove lower boundedness of the constrained functional and verify Palais - Smale condition.  Choosing $q>0$ as large as we need to allow the functional to achieve negative levels, we have all the ingredients to show that minimum is attained for a nontrivial radial function and, at a later stage, that  a multiplicity result analogous to that in\cite{ARu} holds for a possibly larger $q$.
	
	As we will explain in Remark \ref{re:bound}, the situation in the nonradial case is quite different, at least from a technical point of view. We point out that, differently from \cite{CingolaniWeth} where the better known and studied Schr\"odinger-Poisson system 
	
	\beq 
	\left	\{
	\begin{array}{l}
		-\Delta u  +  u +  \phi u -|u|^{p-2}u=0\hbox{ in } \R^2,	\\
		\Delta \phi = u^2 \hbox{ in } \R^2,
	\end{array}
	\right.
	\eeq 
	is treated, our strategy to find nonradial solutions does not rely only on a suitable change of the functional framework. Indeed we will introduce a different approach to \eqref{eq:e01}, transforming the problem  in the following nonlinear eigenvalue type: look for a nonradial $u\neq 0$ and $q>0$ such that
		\begin{equation*}
			-\Delta u +(1+|x|^\a)u + W'(u)=q(\log |x|\star u^2)u.
		\end{equation*}
	In Theorem \ref{main2} we are able to prove the existence of at least a nonradial solution to a class of problem including \eqref{eq:e1} for $p>2$.
	
	We would stress the importance of the role played by the dimension in solving \eqref{eq:e1}, as it falls in the category of problem like \eqref{eq:e01}  with {\it positive potential} (namely $V(x)u^2+W(u)\ge0$). Indeed, existence and multiplicity results we are going to provide in Theorems \ref{main} and \ref{main2} are in contrast to what happens for the analogous three-dimensional version of nonlinear Schr\"odinger-Maxwell system with positive potential. In \cite[Proposition 1.2]{DM} it was proved a general result implying nonexistence of nontrivial solutions to \eqref{eq:e1} in $\RT.$
	\vskip .5cm
	
	Finally, in the last part of the paper we are  interested in studying  \eqref{eq:e01} for other types of nonlinearities, as, for instance, those changing sign. As an example, we will treat the following model problem
	\beq \label{eq:e2}\tag{$\mathcal{P}_2$}
		\left	\{
		\begin{array}{l}
			-\Delta u  +(1+|x|^\a) u- q \phi u -|u|^{p-2}u=0\hbox{ in } \R^2,	\\
			\Delta \phi = u^2 \hbox{ in } \R^2,
		\end{array}
		\right.
		\eeq
	for $\a>0$ and $p>2$. 
	
	This situation, which appears at a first sight simpler to treat because of the formal analogy with \eqref{eq:sm}, presents indeed unexpected difficulties in proving boundedness of Palais - Smale sequences when $p<4$. Similar difficulties have been observed and overcome in \cite{DuWeth} by suitable estimates used inside a contradiction argument. Unfortunately, Du and Weth idea of proof seems to be only partially workable in our situation. Indeed, adapting their estimates to our problem, we are able to get our goal only in the range $2<p<3$, remaining the existence of solutions to \eqref{eq:e2} an open problem for $3\le p <4$.
	\vskip 0,2 cm
	The paper is organized as follows: in Section \ref{sec1} we provide some preliminaries to develop our arguments. Section \ref{sec2} is devoted to the study of \eqref{eq:e1}. It is divided in three parts: in the first, we present a nonexistence result depending on estimate provided in Lemma \ref{le:ineq} and the combination of a large value of $\a$ and a small value of $q$; in the second we look for  existence of radial solutions, providing an existence and multiplicity result due to the choice of a sufficiently large $q$; the third is related with the existence of a nonradial solution to a class of problems, including \eqref{eq:e1}. Finally, Section \ref{sec3} is devoted to the discussion of problem \eqref{eq:e2}.

	\section{Variational framework and some technical results}\label{sec1}
	
	First of all note that, when considering \eqref{eq:e1} and \eqref{eq:e2}, looking for the second equation of each one, one can see that it is quite natural to consider $\phi$ as the newtonian potential of $u^2$, i.e.,
	$$
	\phi(x) = (\Phi_2 \ast u^2)(x) = \frac{1}{2\pi}\int_{\mathbb{R}^2}\log(|x - y|) u^2(y)dy,
	$$
	where $\Phi_2$ is the fundamental solution of the laplacian operator in $\mathbb{R}^2$. Taking this into account, one can see that \eqref{eq:e1} is equivalent to
	\begin{equation}
	-\Delta u + (1+|x|^\alpha)u + q(\Phi_2 \ast u^2) u + |u|^{p-2}u = 0\quad \mbox{in $\mathbb{R}^2$}
	\label{eq:e1e}\tag{$\mathcal{P}_1'$}
	\end{equation}
	and \eqref{eq:e2}, to
	\begin{equation}
	-\Delta u + (1+|x|^\alpha)u + q(\Phi_2 \ast u^2) u - |u|^{p-2}u = 0 \quad \mbox{in $\mathbb{R}^2$}.
	\label{eq:e2e}\tag{$\mathcal{P}_2'$}
	\end{equation}
	In this section we set the variational background to deal with \eqref{eq:e1e} and \eqref{eq:e2e}. First of all let us define
	$$
	X = \left\{u \in H^1(\mathbb{R}^2); \int_{\mathbb{R}^2}(1 + |x|^\alpha)u^2 dx < +\infty \right\},
	$$
	endowed with the norm
	$$
	\|u\| = \left( \|\nabla u\|_2^2 + \int_{\mathbb{R}^2}(1 + |x|^\alpha)u^2 dx \right)^\frac{1}{2}.
	$$
	Let us also denote
	$$
	\|u\|_* = \left(\int_{\mathbb{R}^2}(1 + |x|^\alpha)u^2 dx \right)^\frac{1}{2},
	$$
	in such a way that
	$$
	\|u\| = \left(\|\nabla u\|_2^2 + \|u\|_*^2 \right)^\frac{1}{2}.
	$$
	
	Note that, since $\|u\|_{H^1(\mathbb{R}^2)} \leq \|u\|$ for all $u \in X$, it follows that the following embeddings are continuous,
	\begin{equation}
	X \hookrightarrow L^r(\mathbb{R}^2), \quad \mbox{for all $r \geq 2$}.
	\label{embeddingX}
	\end{equation}

	As in \cite{CingolaniWeth,Stubbe}, let us define
	\begin{eqnarray}
	V_0(u) & = & \frac{1}{2\pi}\int_{\mathbb{R}^2}\int_{\mathbb{R}^2} \mbox{log}(|x - y|)u^2(x)u^2(y)dx dy, \\
	V_1(u) & = & \frac{1}{2\pi}\int_{\mathbb{R}^2}\int_{\mathbb{R}^2} \mbox{log}(2 + |x - y|)u^2(x)u^2(y)dx dy, \\
	V_2(u) & = & \frac{1}{2\pi}\int_{\mathbb{R}^2}\int_{\mathbb{R}^2} \mbox{log}\left(1 + \frac{2}{|x - y|}\right)u^2(x)u^2(y)dx dy\\
	\end{eqnarray}
	and note that $V_0 = V_1 - V_2$, where $V_1, V_2 \geq 0$. Moreover, as in \cite[Section 3.1]{A}, one can prove that there exists $C_\alpha > 1$ such that
	\begin{equation}
	\label{ineq:V1}
	V_1(u) \leq \frac{C_\alpha}{\pi}\|u\|_2^2\|u\|_*^2
	\end{equation}
	and, as in \cite[Section 2]{CingolaniWeth}, there exists $D>0$
	\begin{equation}
	\label{ineq:V2}
	V_2(u) \leq D\|u\|_{\frac 83}^4.
	\end{equation}

	For $u \in X$, we define
	$$
	I_\alpha(u) = \frac{1}{2}\int_{\mathbb{R}^2}|\nabla u|^2 dx + \frac{1}{2}\int_{\mathbb{R}^2}(1 + |x|^\alpha)u^2 dx - \frac{q}{4}V_0(u) + \frac{1}{p}\int_{\mathbb{R}^2}|u|^p dx,
	$$
	and
	$$
	J_\alpha(u) = \frac{1}{2}\int_{\mathbb{R}^2}|\nabla u|^2 dx + \frac{1}{2}\int_{\mathbb{R}^2}(1 + |x|^\alpha)u^2 dx - \frac{q}{4}V_0(u) - \frac{1}{p}\int_{\mathbb{R}^2}|u|^p dx.
	$$
	By \eqref{ineq:V1} and \eqref{ineq:V2}, we can see that $I_\alpha$ and $J_\alpha$ are well defined in $X$. Also, as in  \cite[Lemma 2.2.]{CingolaniWeth}, it follows that $I_\alpha$ and $J_\alpha$ are $C^1$ in this space and their critical points are solutions, respectively,
	of \eqref{eq:e1e} and \eqref{eq:e2e}.
	
	We start with the following fundamental estimate 
		\begin{lemma}\label{le:ineq}
			Let $\a >0$, $p>2$ and $\b>2$ such that the tern $(\a,p,\b)$ satisfies the following inequality $\frac{\a p}{(p-2)(\b-1)}>2,$ and $\eps>0$. Then there exists a positive constant $C$ depending on $(\a,p,\b,\eps)$ such that for any $u \in L^2(\RD, (1+|x|^\a)\,dx)\cap L^p(\RD)$ we have 
				\begin{equation*}
					\left(\ird \log(2+|x|) u^2(x)\, dx\right)^2\le \frac {2\eps} \b \ird (1+|x|^\a) u^2(x)\, dx + C\|u\|_p^{\frac{4(\b-1)}{\b-2}}	.			
				\end{equation*}
		\end{lemma}
		\begin{proof}
			Consider $u$ as in the assumption and take $\a, p,\b$ satisfying the prescribed inequalities. Take $\eps>0$ and compute
				\begin{multline}\label{eq:ineq}
					\ird \log (2+|x|) u^2(x)\, dx =  \ird \frac{\log(2+|x|)}{(1+|x|^\a)^{\frac 1\b}} u^{\frac{2\b-2}{\b}}(x) (1+|x|^\a)^{\frac 1\b}u^{\frac 2 \b}(x)\, dx\\
															 \le \left(\ird \eps(1+|x|^\a) u^2(x)\, dx\right)^{\frac 1\b}\left(\ird \frac{\log^{\frac{\b}{\b-1}}(2+|x|)}{\eps^{\frac 1{\b-1}}(1+|x|^\a)^\frac{1}{\b-1}}u^2(x)\, dx\right)^{\frac{\b-1}{\b}}.
				\end{multline}
			As a consequence, applying the Young inequality, we obtain 
			\begin{align*}
				\bigg(\ird \log (2&+|x|) u^2(x)\, dx\bigg)^2\\
				&\le \frac {2\eps} \b \ird (1+|x|^\a) u^2(x)\, dx + \frac{\b-2}{\b\eps^{\frac 2{\b-2}}}\left(\ird \frac{\log^{\frac{\b}{\b-1}}(2+|x|)}{(1+|x|^\a)^\frac{1}{\b-1}}u^2(x)\, dx\right)^{\frac{2(\b-1)}{\b-2}}\\
				&\le \frac {2\eps} \b \ird (1+|x|^\a) u^2(x)\, dx \\
				&\qquad +\frac{\b-2}{\b\eps^{\frac 2{\b-2}}} \left(\ird \frac{\log^{\frac{\b p}{(\b-1)(p-2)}}(2+|x|)}{(1+|x|^\a)^\frac{ p}{(\b-1)(p-2)}}\, dx\right)^{\frac{2(\b-1)(p-2)}{(\b-2)p}}\|u\|_p^{\frac{4(\b-1)}{\b-2}}.
			\end{align*}
		We conclude observing that in the last line the integral converges.
		\end{proof}
		
		Now, let us prove a technical result involving $V_0$.
		\begin{lemma}\label{le:weaklycontinuous}
			$V_0$ is weakly continuous in $X$.\\
			As a consequence, since $\langle V'_0(u),u\rangle=4 V_0(u)$ for all $u\in X$, also the map
			$$u\in X\mapsto \langle V'_0(u),u\rangle\in \R$$
			is weakly continuous.
		\end{lemma}
		\begin{proof}
		Let $(u_n) \subset X$ and $u \in X$ such that
		$$
		u_n \rightharpoonup u \quad \mbox{in $X$.}
		$$
		By the coerciveness of the potential $x \mapsto \left(1 + |x|^\alpha\right)$, the embedding $X \hookrightarrow L^r(\mathbb{R}^2)$ is compact, for all $r \geq 2$. Then,
		$$
		u_n \to u \quad \mbox{in $L^r(\mathbb{R}^2)$, for all $r \geq 2$.}
		$$
		Then
		\begin{eqnarray}
		\nonumber |V_1(u_n) - V_1(u)| & = & \int_{\mathbb{R}^2}\int_{\mathbb{R}^2}\mbox{log}\left(2 + |x-y|\right)u_n^2(x)\left(u_n^2(y) - u^2(y)\right) dxdy\\
		\nonumber & & + \int_{\mathbb{R}^2}\int_{\mathbb{R}^2}\mbox{log}\left(2 + |x-y|\right)u^2(x)\left(u_n^2(y) - u^2(y)\right) dxdy\\
		\nonumber & \leq & \int_{\mathbb{R}^2}\int_{\mathbb{R}^2}\mbox{log}\left(2 + |x|\right)u_n^2(x)\left|u_n(y) - u(y)\right| \left|u_n(y) + u(y)\right| dxdy\\
		\nonumber & & + \int_{\mathbb{R}^2}\int_{\mathbb{R}^2}\mbox{log}\left(2 + |y|\right)u^2_n(x)\left|u_n(y) - u(y)\right| \left|u_n(y) + u(y)\right| dxdy\\
		\nonumber & & + \int_{\mathbb{R}^2}\int_{\mathbb{R}^2}\mbox{log}\left(2 + |x|\right)u^2(x)\left|u_n(y) - u(y)\right| \left|u_n(y) + u(y)\right| dxdy\\
		\label{eq:wc1} & & + \int_{\mathbb{R}^2}\int_{\mathbb{R}^2}\mbox{log}\left(2 + |y|\right)u^2(x)\left|u_n(y) - u(y)\right| \left|u_n(y) + u(y)\right| dxdy.
		\end{eqnarray}
		Let us denote $C_\alpha, C_\alpha' > 0$, such that
		\begin{equation}
		\mbox{log}(2+r) \leq C_\alpha(1+r^\alpha) \quad \mbox{and} \quad \mbox{log}^2(2+r) \leq C_\alpha'(1+r^\alpha),
		\label{eq:wc2}
		\end{equation}
		for all $r > 0$. Then, \eqref{eq:wc1}, \eqref{eq:wc2} and H\"older inequality imply that
		\begin{eqnarray*}
		|V_1(u_n) - V_1(u)| & \leq & C_\alpha \int_{\mathbb{R}^2}\int_{\mathbb{R}^2}\left(1 + |x|^\alpha\right)u_n^2(x)\left|u_n(y) - u(y)\right| \left|u_n(y) + u(y)\right| dxdy\\
		& & + \|u_n\|_2^2 \int_{\mathbb{R}^2}\mbox{log}^2\left(2 + |y|\right)\left|u_n(y) + u(y)\right|^2 dy \|u_n - u\|_2^2\\
		& & + C_\alpha\int_{\mathbb{R}^2}\int_{\mathbb{R}^2}\left(1 + |x|^\alpha\right)u^2(x)\left|u_n(y) - u(y)\right| \left|u_n(y) + u(y)\right| dxdy\\
		& & + \|u\|_2^2 \int_{\mathbb{R}^2}\mbox{log}^2\left(2 + |y|\right)\left|u_n(y) + u(y)\right|^2 dy \|u_n  - u\|_2^2\\
		& \leq & C_\alpha \|u_n\|_*^2 \|u_n - u\|_2^2 \|u_n + u\|_2^2 + C_\alpha'\|u_n\|_2^2\|u_n - u\|_2^2 \|u_n + u\|_*^2\\
		& & C_\alpha \|u\|_*^2 \|u_n - u\|_2^2 \|u_n + u\|_2^2 + C_\alpha'\|u\|_2^2\|u_n - u\|_2^2 \|u_n + u\|_*^2\\
		& = & o_n(1).
		\end{eqnarray*}
		
		Hence $V_1$ is weakly continuous. Moreover, by \cite[Lemma 2.2]{CingolaniWeth}, $V_2$ is also weakly continuous. By these facts, it follows that
		$$
		V_0(u_n) \to V_0(u),
		$$
		as $n \to +\infty$.
		\end{proof}

\section{Existence and nonexistence results to \eqref{eq:e1} for $p>4$} \label{sec2}

\subsection{A nonexistence result}

We start this subsection by proving a nonexistence result to \eqref{eq:e1}, for $p > 4$. More specifically, we are going to prove the following result.

\begin{theorem}\label{th:nonex}
	Let $p>4$ and $\a>\tilde p$. Then there exists a constant $\bar q>0$ such that if $q\in(0,\bar q)$, then \eqref{eq:e1} (or equivalently \eqref{eq:e1e}) does not have any nontrivial solution.
\end{theorem}
\begin{proof}
	Assume $u \in X$ is a solution of \eqref{eq:e1e}. Then 
	\begin{align*}
	0&=\ird |\n u(x)|^2\, dx + \ird (1+|x|^\a)u^2(x)\, dx\\
	&\qquad- \frac q{2\pi} \ird \ird \log (2+|x-y|)u^2(x)u^2(y)\, dx dy\\
	&\qquad+\frac q{2\pi}\ird\ird \log\bigg(1+\frac{2}{|x-y|}\bigg)u^2(x)u^2(y)\,dxdy +\ird |u(x)|^p\,dx\\
	&\ge \ird |\n u(x)|^2\, dx + \ird (1+|x|^\a)u^2(x)\, dx\\
	&\qquad- \frac q{\pi} \ird \ird \log (2+|x|)u^2(x)u^2(y)\, dx dy +\ird |u(x)|^p\,dx\\
	&\ge \ird |\n u(x)|^2\, dx + \ird (1+|x|^\a)u^2(x)\, dx\\
	&\qquad- \frac q{\pi\log 2} \left( \ird \log (2+|x|)u^2(x)\, dx\right)^2 +\ird |u(x)|^p\,dx.
	\end{align*}
	By Lemma \ref{le:ineq} applied for $\b=\frac{2p-4}{p-4}$ and $\eps=1$ we can proceed as follows
	\begin{align*}
	0&\ge \ird |\n u(x)|^2\, dx + \ird (1+|x|^\a)u^2(x)\, dx\\
	&\qquad- q C_1 \ird (1+|x|^\a)u^2(x)\, dx -q C_2 \|u\|_p^p + \|u\|_p^p.
	\end{align*}
	Choosing  $\bar q =\min \left(\frac{1}{C_1},\frac 1{C_2}\right)$, we easily conclude that $u=0$.
\end{proof}

\subsection{Existence of radial solutions}

In this subsection, we are going to study the problem in the set of radial functions. Following the very original approach developed by Ruiz \cite[Section 4]{Ruiz}, we again would like to emphasize the analogies arising between the three dimensional Schr\"odinger-Maxwell system perturbed by the nonlinear local term $-|u|^{p-2}u$, and this two dimensional model of the Schr\"odinger-Maxwell system, in presence of the reversed sign local nonlinearity  $|u|^{p-2}u$.  

We introduce the space
$$
X_r = \{u \in X; \, u(x) = u(|x|)\},
$$
which, as is well known, is a natural constraint.

The following Strauss type radial Lemma will be a useful tool for subsequent estimates 
\begin{lemma}\label{le:str}
	Let $\a>0$. Then every $u\in X_r$ is almost everywhere equal to a continuous function $U \in \RD\setminus\{0\}$. Moreover, there exists a constant $\tilde C>0$ uniform with respect to $u\in X_r$ such that 
	\begin{equation*}
	|u(x)|\le \frac{\tilde C}{|x|^{\frac{\a+2}4}}\|u\|, \quad \hbox{ for }|x|\ge 1.
	\end{equation*}
	\begin{proof}
		Since the first part is standard, we only prove the estimate. 
		
		Let $k\le \frac{\a+2}2$ and consider $u$ a radial function in $C_0^\infty(\RD)$.
		For any $r\ge 0$, we have that 
		\begin{align*}
		\left|\frac{d}{dr}\left(r^ku^2(r)\right)\right|&\le k r^{k-1}u^2(r)+2r^k|u(r)||u'(r)|\\
		&\le k r^{k-1}u^2(r) + r^{2k-1}u^2(r) + r|u'(r)|^2. 
		\end{align*}
		Now, fix $r\ge 1$ and integrate $-\frac{d}{ds}\left(s^ku^2(s)\right)$ in the interval $[r,+\infty)$. We have
		
		\begin{align*}
		r^ku^2(r)&\le k\int_r^{+\infty} s^{k-2-\a}s s^\a u^2(s)\, ds+\int_r^{+\infty}s^{2k-2-\a}ss^\a u^2(s)\, ds+ \frac{\|\n u\|_2^2} {2\pi}\\
		&\le \frac k{\sqrt{2\pi}} r^{k-2-\a}\ird (1+|x|^\a) u^2\, dx+\frac {r^{2k-2-\a}}{\sqrt{2\pi}}\ird (1+|x|^\a) u^2\, dx+\frac{\|\n u\|_2^2} {2\pi}.
		\end{align*}
		The conclusion follows taking $k=\frac{\a+2}2$ and then by a density argument.
	\end{proof}
	
\end{lemma}

\begin{lemma}\label{th:bound}
	For any $\a>\tilde p$ there exists $\tilde q >0$ such that, for every $q\ge \tilde q$ we have $\inf_{u\in X_r} I_\a(u)\in (-\infty,0)$.
\end{lemma}
\begin{proof}
	First of all, let us consider $\tilde q$ large enough in such a way that $\inf_{u\in X_r} I_\a(u)<0$ (such a constant $\tilde q$ exists by \eqref{eq:divergence}). For $q\ge \tilde q$, let us show that 
	$$\inf_{u\in X_r} I_\a(u)\neq -\infty.$$
	
	Using Young's and H\"older's inequalities as in \eqref{eq:ineq}, we obtain, for $\b >2$ and $\eps>0$ 
	\begin{align*}
	\bigg(\ird \log (2&+|x|) u^2(x)\, dx\bigg)^2\\
	&\le \frac {2\eps} \b \ird (1+|x|^\a) u^2(x)\, dx + \frac{\b-2}{\b\eps^{\frac 2{\b-2}}}\left(\ird \frac{\log^{\frac{\b}{\b-1}}(2+|x|)}{(1+|x|^\a)^\frac{1}{\b-1}}u^2(x)\, dx\right)^{\frac{2(\b-1)}{\b-2}}\\
	&\le \frac {2\eps} \b \ird (1+|x|^\a) u^2(x)\, dx \\
	&\qquad +\frac{\b-2}{\b\eps^{\frac 2{\b-2}}} \left(\ird \frac{\log^{2}(2+|x|)}{(1+|x|^\a)^\frac{ 2}{\b}}\, dx\right)^{\frac{\b}{\b-2}}\ird |u|^{\frac{4(\b-1)}{\b-2}}(x)\,dx.
	\end{align*}
	Assume that $\a > \b$ in such a way that
	\begin{equation*}
	\ird \frac{\log^{2}(2+|x|)}{(1+|x|^\a)^\frac{ 2}{\b}}\, dx<+\infty.
	\end{equation*}
	Then in a similar way as in Theorem \ref{th:nonex}
	\begin{align*}
	I_a(u)&\ge \frac{1}{2}\int_{\mathbb{R}^2}|\nabla u|^2 dx + \frac{1}{2}\int_{\mathbb{R}^2}(1 + |x|^\alpha)u^2 dx - \frac{q}{4}V_0(u) + \frac{1}{p}\int_{\mathbb{R}^2}|u|^p dx\\
	&\ge \frac{1}{2}\int_{\mathbb{R}^2}|\nabla u|^2 dx + \frac{1}{2}\int_{\mathbb{R}^2}(1 + |x|^\alpha)u^2 dx+\frac q4V_2(u)\\
	&\qquad  - \frac q{4\pi\log 2} \left( \ird \log (2+|x|)u^2(x)\, dx\right)^2 + \frac{1}{p}\int_{\mathbb{R}^2}|u|^p dx\\
	&\ge \frac{1}{2}\int_{\mathbb{R}^2}|\nabla u|^2 dx +\left(\frac 14 -\frac{q\eps}{2\pi\b\log 2 }\right) \ird (1+|x|^\a)u^2\, dx \\
	&\qquad+\frac q4V_2(u) +\ird \left(\frac 1 4 u^2  - C_\eps |u|^{\frac{4(\b-1)}{\b-2}}+\frac 1 p |u|^p\right)\, dx.
	\end{align*}
	
	Now, since the functional has to be coercive in its principal part, we choose $\eps>0$ such that $\frac 14 -\frac{q\eps}{2\pi\b\log 2 }>\frac 18.$ Moreover, since $\a>\tilde p$, we can choose $\beta$ sufficiently close to $\a$ in such a way $\frac{4(\b-1)}{\b-2}<p$. On the other hand, it is immediate to see that $4<\frac{4(\b-1)}{\b-2}.$
	
	Now  we prove that the functional $H$ in $X_r$ defined by 
	\begin{align*}
	H(u)&=\frac{1}{2}\int_{\mathbb{R}^2}|\nabla u|^2 dx +\left(\frac 14 -\frac{q\eps}{2\pi\b\log 2 }\right) \ird (1+|x|^\a)u^2\, dx \\
	&\qquad+\frac q4V_2(u) +\ird \left(\frac 1 4 u^2  - C_\eps |u|^{\frac{4(\b-1)}{\b-2}}+\frac 1 p |u|^p\right)\, dx
	\end{align*}
	is bounded from below.
	
	We define  $f:[0,+\infty)\to\R$ such that 
	\begin{equation*}
	f(s)=\frac 1 4 s^2  - C_\eps s^{\frac{4(\b-1)}{\b-2}}+\frac 1 p s^p
	\end{equation*}
	and set $m=\min_{s\ge 0}f(s)$. Assume $m<0$ (otherwise we are already done) and consider the open set $N=\{s>0\mid f(s)<0\}.$ A simple study of $f$ shows that $N$ is an open bounded interval $(\g,\d)$ with $\g>0$.
	
	We proceed with the following lower estimate on $H$
	
	\begin{align}\label{eq:estimateonH}
	H(u)&\ge \ird \frac 14 |\n u|^2\, dx+\frac 18 (1+|x|^\a)u^2 \, dx+\frac q4V_2(u) +\ird f(u)\, dx\\
	&\ge \ird \frac 14 |\n u|^2\, dx+\frac 18 (1+|x|^\a)u^2 \, dx+\frac q4V_2(u) +\int_{\{x\in\RD\mid|u(x)|\in (\g,\d)\}} f(u)\, dx\nonumber\\
	&\ge \ird \frac 14 |\n u|^2\, dx+\frac 18 (1+|x|^\a)u^2 \, dx+\frac q4V_2(u) +m|A_u|\nonumber
	\end{align}
	where $A_u=\{x\in\RD\mid|u(x)|\in (\g,\d)\}$ and $|A_u|$ denotes the Lebesgue measure of $A_u$.
	
	Assume by contradiction that there exists a sequence $(u_n)_{n \in \N}$ in $X_r$ such that $\lim_{n \to +\infty}H(u_n) = -\infty.$  By this fact, and since we can assume that up to a subsequence $H(u_n)<0$ for all $n\ge 0$, we have
	\begin{itemize}
		\item[1.] $\lim_n |A_n|=+\infty$,
		\item[2.] $\frac 1 8 \|u_n\|^2\le -m |A_n|$, for all $n\ge 1$,
		\item[3.] $\frac q 4 V_2(u_n) \le -m|A_n|$, for all $n\ge 1$,
	\end{itemize} 
	where $A_n = A_{u_n}$.
	
	Take $n\ge 1$ and set $\rho_n=\sup_{x\in A_n}|x|$. Since $|A_n|>0$, certainly $\rho_n>0$ and, by Lemma \ref{le:str}, $\rho_n=\max_{x\in A_n}|x|\in\R$.
	
	By $3.$, we have 
	\begin{align*}
	|A_n|&\ge -\frac q{4m}V_2(u_n)\ge -\frac{q}{8m\pi}\int_{\mathbb{R}^2}\int_{\mathbb{R}^2} \mbox{log}\left(1 + \frac{2}{|x - y|}\right)u^2(x)u^2(y)dx dy\\
	&\ge -\frac{q}{8m\pi} \int_{A_n}\int_{A_n}\mbox{log}\left(1 + \frac{2}{|x - y|}\right)u^2(x)u^2(y)dx dy\\
	&\ge -\frac{q\g^4}{8m\pi}\mbox{log}\left (1+\frac 1\rho_n\right)|A_n|^2,
	\end{align*}
	and then  there exists $c_1>0$ such that
	\begin{equation}\label{eq:bound1}
	\mbox{log}\left (1+\frac 1\rho_n\right)|A_n|\le c_1.
	\end{equation}
	Now, if $(\rho_n)_n$ is bounded from above, we easily get a contradiction comparing \eqref{eq:bound1} with $1.$ Assume then that $\lim_n\rho_n=+\infty$. Define $x_n\in A_n$ such that $|x_n|=\rho_n$.  By $2.$, Lemma \ref{le:str} and since we can assume $|x_n|\ge 1$ for every $n\ge 1$,
	\begin{equation*}
	\g \le |u_n(x_n)|\le\frac{\tilde C}{|x_n|^{\frac{\a+2}4}}\|u\|\le \frac{\sqrt{-8m} \tilde C}{\rho_n^{\frac{\a + 2}4}} |A_n|^{\frac 12}
	\end{equation*}
	and then there exists $c_2>0$ such that
	\begin{equation}\label{eq:bound2}
	\rho_n^{\frac{\a + 2}2}\le c_2 |A_n|.
	\end{equation}
	Comparing \eqref{eq:bound1} with \eqref{eq:bound2} we have
	\begin{equation*}
	\rho_n^{\frac{\a + 2}2}\mbox{log}\left (1+\frac 1\rho_n\right)\le c_1c_2
	\end{equation*}
	and, since $\a>0$, this contradicts the fact that $\lim_n\rho_n=+\infty$.
\end{proof}

\begin{lemma}\label{le:palaissmale}
	The functional $I_\a$ restricted to $X_r$ satisfies the Palais-Smale condition.
\end{lemma}
\begin{proof}
	Assume that $(u_n)_n$ is a Palais-Smale sequence for the functional ${I_\a}_{|X_r}$. In particular, we have that, for any $n$ large enough, 
	
	$$\langle I'_\a(u_n),u_n\rangle\le|\langle I'_\a(u_n),u_n\rangle|\le \|u_n\|$$
	that is
	\begin{equation}\label{ps}
	\|u_n\|^2 - q V_0(u_n) + \|u_n\|_p^p\le \|u_n\|.
	\end{equation}
	On the other hand, making computations similar to those in Lemma \ref{th:bound}, we have that for any $\eps>0$ there exists $C_\eps>0$ such that
	\begin{align}\label{pss}
	&\|u_n\|^2 - q V_0(u_n) + \|u_n\|_p^p\\
	&\qquad\qquad\ge \ird |\n u_n|^2\, dx + \left(\frac 12 -\frac{2q\eps}{\pi\beta\log 2}\right) \ird (1+|x|^\a)u_n^2\,dx\nonumber\\
	&\qquad\qquad \qquad +qV_2(u_n)+ \ird \left(\frac 12 u^2- C_\eps |u_n|^{\frac{4(\b-1)}{\b-2}}+|u_n|^p \right)  dx.\nonumber
	\end{align}

	Comparing \eqref{ps} and \eqref{pss}, for a suitable choice of $\eps>0$, we obtain 
	\begin{equation*}
	\frac 14 \|u_n\|^2+q V_2(u_n) + \ird \left(\frac 12 u^2- C_\eps |u_n|^{\frac{4(\b-1)}{\b-2}}+|u_n|^p \right)  dx\le \|u_n\|.
	\end{equation*}
	We claim that $(\|u_n\|)_n$ is bounded. Otherwise, for any large value of $n$, we would have
	\begin{equation*}
	\frac 18 \|u_n\|^2 + qV_2(u_n) +\ird \left(\frac 12 u^2- C_\eps |u_n|^{\frac{4(\b-1)}{\b-2}}+|u_n|^p \right)  dx\le \|u_n\|-\frac 18 \|u_n\|^2\le 0
	\end{equation*}
	which leads to a contradiction by the same arguments developed in the proof of Lemma \ref{th:bound} from inequality \eqref{eq:estimateonH} on. Now, as usual, we extract a subsequence, relabeled in the same way, weakly convergent to $\bar u\in X_r$. At this point we can conclude as usual, once we have observed that
	\begin{align*}
	\|u_n\|^2 &= q V_0(u_n) - \|u_n\|_p^p + o_n(1)\\
	&= q V_0(\bar u) - \|\bar u\|_p^p + o_n(1)\\
	&= \|\bar u\|^2 + o_n(1)
	\end{align*}
	where we have used 
	$$\langle I'_\a(u_n),u_n\rangle=o_n(1)$$
	in the first equality, compactness in the second and
	\begin{equation*}
	0=\lim_n \langle I'_\a(u_n),\bar u\rangle=\langle I'_\a(\bar u), \bar u\rangle
	\end{equation*}
	in the third.
\end{proof}
\begin{theorem}\label{main}
	Let $\a>\tilde p$ and $q\ge \tilde q$. Then the system \eqref{eq:e1} has both a minimizer (at a negative level) and a mountain pass solution on the set of radial functions. 
	
	Moreover, there exists a nonincreasing sequence $(q_n)_{n\ge 1}$ such that, for any $n\ge 1$, system \eqref{eq:e1}  possesses at least $n$ couples of radial solutions $(\pm u_{k,q})$, $k=1,\ldots, n$, such that $I_\a(u_{k,q})<0$ and $n$ couples of radial solutions $(\pm v_{k,q})$, $k=1,\ldots, n$, such that $I_\a(v_{k,q})>0$, for any $q\ge q_n$.
\end{theorem}
\begin{proof}
	Using the Ekeland variational principle, the existence of a minimizer follows directly from Lemma \ref{th:bound} and Lemma \ref{le:palaissmale} by a well known variant of Weierstrass Theorem. 
	
	To show that there exists a mountain pass solution, by Lemma \ref{le:palaissmale}, we only have to check the geometrical conditions of Ambrosetti-Rabinowitz mountain pass Theorem.\\ 
	Observe that, if we take any $u\in X_r$ such that $\|u\|=\rho$ with $\rho>0$ small enough, then by \eqref{embeddingX} and \eqref{ineq:V1} 
	\begin{align*}
	I_\a(u) &\ge  \frac{1}{2}\|u\|^2 - C_1 \|u\|_2^2\|u\|_*^2 + \frac 1 p \|u\|_p^p\\
	&\ge \frac{1}{2}\|u\|^2 - C_1 \|u\|^4 =\rho^2 \left(\frac 12 -C_1\rho^2\right)\\
	&\ge \frac{\rho^2}4.
	\end{align*} 
	Moreover the second mountain pass geometric condition is trivially satisfied by the minimizer $u_0$.
	
	In order to prove the second part of the theorem, we will use abstract results of \cite{AR}, proceeding in a similar way as in \cite[Theorem 3.1]{ARu}.\\
	First we look for solutions at negative levels of the functional. Consider an increasing (in the sense of inclusion) sequence $(X_n)_n$, such that for any $n\ge 1:$ $X_n$ is an $n-$dimensional subspace of $X_r$.   Set $S=\{u\in X_r;\|u\|=1\}$ and $S_n=X_n\cap S$ for any $n\ge 1$.
	
	We define
	\begin{equation*}
	\begin{array}{lll}
	&M_1=\displaystyle\max_{u\in S_n}\frac 12 \ird (|\n u|^2+(1+|x|^\a)u^2)\, dx,\quad &m_2=\displaystyle \min_{u\in S_n}\frac 1{8\pi}\left(\ird u^2\, dx\right)^2 ,\\
	&\displaystyle M_3=\max_{u\in S_n}\frac 14 |V_0(u)|,\quad &M_4=\displaystyle\max_{u\in S_n}\frac 1 p\ird |u|^p\, dx,
	\end{array}
	\end{equation*}
	and, for all $u\in X_r$, set $u_t=u(\cdot/t)$. Consider $\bar t>1$ such that $ m_2\log \bar t  - M_3 >0$ :
	\begin{align}\label{eq:mult}
	I_\a(u_{\bar t})&= \frac12\|\n u\|_2^2 + \frac{\bar t^2}2\ird (1+\bar t^\a|x|^a)u^2\, dx\\ 
	&\qquad - \frac{q\bar t^4\log\bar t}{8\pi}\left(\ird u^2\, dx \right)^2 - \frac{q\bar t^4}{4} V_0(u) + \frac{\bar t^2}{p}\ird |u|^p\, dx\nonumber\\
	&\le M_1\bar t^{2+\a} - q( m_2\log \bar t  - M_3 )\bar t^4  + M_4\bar t^2.\nonumber
	\end{align}
	We take $q^{(1)}_n>0$ such that $M_1\bar t^{2+\a} - q^{(1)}_n( m_2\bar t^4\log \bar t  - M_3 \bar t^4)  + M_4\bar t^2<0$ and define $\tau: X_r\to X_r$ such that 	$\tau (u)=u_{\bar t}$. Of course, $\tau$ is continuous and odd and then by well known properties coming from index theory,  we  have $	\g(\tau(S_n))\ge n$, where $\g$ denotes the Krasnoselski genus.
	
	Now take $q\ge q^{(1)}_n$. By \eqref{eq:mult} we deduce that, for any $k=1,\ldots, n$,
	\begin{equation*}
	d_k = \inf_{\g (A)\ge k \atop A\in\Sigma (X_r)}\sup\{I_\a(u);u\in A\} <0
	\end{equation*}
	where $\Sigma (X_r)$ represents the set of closed subsets of $X_r\setminus\{0\}$ which are symmetric with respect to the origin. Since $I_\a$ satisfies the Palais-Smale condition and $I_\a$ is bounded from below, by \cite[Corollary 2.24]{AR} every value $d_k$ is critical, and there exists at least $n$ couples of solutions, some of them possibly at the same level. \\
	As for the (at least) $n$ couples of solutions at positive levels, we just have to verify   condition $I_7$ to apply \cite[Theorem 2.23]{AR}.
	Consider $n\ge 1$ and  $u_1,\ldots, u_n\in X_r$ linearly independent. Define $S^{n-1}=\{\sigma\in\R^n; |\sigma|=1\}$  and for any $\sigma=(\sigma_1,\ldots,\sigma_n) \in S^{n-1}$ set $$u_\sigma=\sum_{i=1}^n\sigma_iu_i.$$ 
	Moreover, we define
	\begin{equation*}
	\begin{array}{lll}
	&M_1=\displaystyle\max_{\sigma\in S^{n-1}}\frac 12 \ird (|\n u_\sigma|^2+(1+|x|^\a)u_\sigma^2)\, dx,\quad &m_2=\displaystyle \min_{\sigma\in S^{n-1}}\frac 1{8\pi}\left(\ird u_\sigma^2\, dx\right)^2 ,\\
	&\displaystyle M_3=\max_{\sigma\in S^{n-1}}\frac 14 |V_0(u_\sigma)|,\quad &M_4=\displaystyle\max_{\sigma\in S^{n-1}}\frac 1 p\ird |u_\sigma|^p\, dx,
	\end{array}
	\end{equation*}
	and $u_{\sigma t}= u_\sigma(\cdot/ t)$. Consider $\tilde t>1$ such that $ m_2\log \tilde t  - M_3 >0$  and take $q^{(2)}_n>0$ such that
	\begin{align*}
	I_\a(u_{\sigma \tilde t})&= \|\n u_\sigma\|_2^2 + \frac{\tilde t^2}2\ird (1+\tilde t^\a|x|^a)u_\sigma^2\, dx\\ 
	&\qquad - \frac{q^{(2)}_n\tilde t^4\log t}{8\pi}\left(\ird u_\sigma^2\, dx \right)^2 - \frac{q^{(2)}_n\tilde t^4}{4} V_0(u_\sigma) + \frac{\tilde t^2}{p}\ird |u_\sigma|^p\, dx\\
	&\le M_1\tilde t^{2+\a} - q^{(2)}_n(m_2\log t  - M_3) \tilde t^4  + M_4\tilde t^2<0.
	\end{align*}
	Observe that for any $q\ge q^{(2)}_n$, the set $A\{u_\s\in X_r; \s\in S^{n-1}\}$ verifies $I_7$ of \cite[Theorem 2.23]{AR}.
	
	We conclude our proof taking $q_n=\max \{q^{(1)}_n, q^{(2)}_n\}$ for all $n\ge 1$. 
\end{proof}
	\begin{remark}\label{re:bound}
		Observe that the minimizing argument, and in particular the proof of Lemma \ref{th:bound}, can not be repeated excluding the property of radiality. In \cite{Ruiz2} it was showed that radiality assumption is not just technical in the three dimensional case, since the author proved that, on the other hand, the unconstrained functional turns out to be unbounded from below. Unfortunately, differently from \cite{Ruiz2}, we are not able to obtain an analogous result in our situation, so that the question about global boundedness and existence of a global minimizer remains an interesting open problem. 
	\end{remark}

\subsection{Existence of a nonradial solution by constrained minimization}

It is worthy of note that the arguments previously used to prove the existence of solutions work by the assumption of radial symmetry on the functional framework, necessary to exploit Strauss estimate. To complete our study on problem \eqref{eq:e1}, we are going to show the existence of a nonradial solution.

Unfortunately, since we will proceed by means of constrained minimization, the result we present is weaker than Theorem \ref{main}, due to the fact that we treat $q$ as an unknown of the problem, coming in the form of a Lagrange multiplier.

First of all, in order to prevent our solution to be radial, we change our functional framework from $X_r$ to 
\begin{equation*}
\tilde X = \{u\in X ; u(-x_1,x_2)=-u(x_1,x_2) \hbox{ and } u(x_1,-x_2)=u(x_1,x_2),\, \forall (x_1,x_2)\in\RD\}.
\end{equation*}
Such a space has been introduced in \cite{DuWeth} in order to prove the existence of nonradial solutions to the planar Schr\"odinger-Poisson system. By the well known Palais principle of symmetric criticality, $\tilde X$ is a natural constraint. We are going to prove our result in a more general setting. Indeed, assume that $W:\R\to\R$ satisfies the following assumptions
\begin{itemize}
	\item [$W1)$] $W=W(s)\in C^1(\R)$,
	\item [$W2)$] $W$ is nonnegative,
	\item [$W3)$] there exist $p>2$, $C_1$ and  $C_2$ positive constants such that $W(s)\le C_1s^2+C_2|s|^p$,
\end{itemize}

\begin{theorem}\label{main2}
	Assume $W1, W2$ and $W3$ and consider the problem
	\beq \label{eW}\tag{${\mathcal{P}}_{W}$}
	\left	\{
	\begin{array}{l}
		-\Delta u +(1+|x|^\a)u -q \phi u +W'(u)=0\hbox{ in } \R^2,	\\
		\Delta \phi = u^2 \hbox{ in } \R^2,
	\end{array}
	\right.
	\eeq 
	Then for any $\a>0$ there exists $q >0$ such that $\eqref{eW}$ has a  nonradial solution. Such a solution is odd with respect to the first variable and even with respect to the second.
\end{theorem}
\begin{proof}
	Let us consider the functional 
	\begin{equation*}
	G_\alpha(u) = \frac 1 2 \ird |\n u|^2\, dx  + \frac 1 2 \ird (1+|x|^\a)u^2\, dx +\ird W(u)\, dx
	\end{equation*}
	constrained to the set
	\begin{equation}
	H:=\left\{u\in  \tilde X; \, V_0(u)=1\right\}.
	\end{equation}
	It is straightforward to see that $G_\alpha$ is a $C^1$ functional in $\tilde X$. Also, $H$ is a $C^1$ manifold in $X$, since for every $u\in H$, by \cite[Lemma 2.3]{DuWeth}, $V \in C^1(X)$ and also $\langle V'_0(u), u\rangle= 4V_0(u)=4$. Moreover, $H$ is also a non empty set. Indeed, for $u\in X\setminus \{0\}$, we consider $u_t=u(\cdot/t)$ and compute
	\begin{align}\label{eq:rescaling}
	V_0(u_t)&= \frac{1}{2\pi}\int_{\mathbb{R}^2}\int_{\mathbb{R}^2} \mbox{log}(|x - y|)u_t^2(x)u_t^2(y)dx dy\\
	&=\frac{t^4}{2\pi}\int_{\mathbb{R}^2}\int_{\mathbb{R}^2} \mbox{log}(|x - y|)u^2(x)u^2(y)dx dy\nonumber\\
	&\qquad + \frac{t^4\log t}{2\pi}\left(\ird u^2(x)\, dx\right)^2.\nonumber
	\end{align}
	Since 
	\begin{equation}\label{eq:divergence}
	\lim_{t\to 0 }V_0(u_t)=0\hbox{ and } \lim_{t\to+\infty }V_0(u_t)=+\infty,
	\end{equation}
	it follows that there exists $t_0 > 0$ such that $t_0u \in H$. 
	
	Coerciveness of $G$ is straightforward to see, since
	$$
	G_\alpha(u) \geq \frac{1}{2}\|u\|^2.
	$$
	In order to show that there exists $\overline{u} \in H$ such that
	$$
	G_\alpha(\overline{u}) = \inf_{u \in H}G_\alpha(u),
	$$
	let $(u_n) \subset H$ be such that
	\begin{equation}
	\lim_{n \to +\infty}G_\alpha(u_n) =\inf_{u \in H}G_\alpha(u).
	\label{eq:galpha1}
	\end{equation}
	By the coercivity of $G_\alpha$, it follows that $(u_n)$ is bounded in $\tilde X$. Then, there exists $\overline{u} \in \tilde X$ such that, up to a subsequence,
	\begin{equation}
	u_n \rightharpoonup \overline{u} \quad \mbox{in $\tilde X$, as $n \to +\infty$.}
	\label{eq:galpha2}
	\end{equation}
	From Lemma \ref{le:weaklycontinuous}, $\overline{u} \in H$. Also, from the compactness of the embeddings $\tilde X \hookrightarrow L^r(\mathbb{R}^2)$, for all $r \geq 2$, it follows that
	\begin{equation}
	u_n \to \overline{u} \quad \mbox{in $L^r(\mathbb{R}^2)$, as $n \to +\infty$.}
	\label{eq:galpha3}
	\end{equation}
	In particular, 
	\begin{equation}
	\|u_n\|_p^p \to \|\overline{u}\|_p^p, \quad \mbox{as $n \to +\infty$.}
	\label{eq:galpha4}
	\end{equation}
	Hence, it follows that
	\begin{eqnarray*}
		\inf_{u \in H}G_\alpha(u) & \leq & G_\alpha(\overline{u})\\
		& \leq & \liminf_{n \to +\infty}\left(\frac{1}{2}\|u_n\|^2 + \frac{1}{p}\|u_n\|_p^p\right)\\
		& = & \liminf_{n \to +\infty}G_\alpha(u_n)\\
		& = & \inf_{u \in H}G_\alpha(u),
	\end{eqnarray*}	
	from where it follows that $\overline{u}$ is a minimizer of $G_\alpha$ on $H$.			
	
	Then, there exists a Lagrange multiplayer $\lambda \in \R$ such that 
	\begin{equation}\label{eq:lag}
	G_\alpha'(\overline{u})=\l V_0'(\overline{u}).
	\end{equation}
	In order to prove that $\l>0$, we recall an argument used in \cite[Page 327]{BL}.  First observe that $\l\neq 0$. Indeed, if we assume $\l=0$, then $\bar u$ would be a nontrivial solution of the equation
		\begin{equation*}
				-\Delta u +(1+|x|^\a)u  +W'(u)=0\hbox{ in } \R^2,
		\end{equation*} 
	However, by standard arguments, we could show that $\bar u$ satisfies the following Pohozaev identity
		\begin{equation*}
			\|\bar u\|_2^2+\left(\frac{2+\a} 2\right)\ird |x|^\a \bar u^2\,dx+ 2\ird W(\bar u)\, dx=0
		\end{equation*}
	which, by assumption $W2$, is satisfied only by $\bar u=0$.\\
	So, assume by contradiction that $\l<0$. Then, since $\langle V'_0(\bar u),\bar u\rangle>0$, by \eqref{eq:lag} certainly $\langle G'_\a(\bar u),\bar u\rangle<0$. Choosing $\eps>0$ small enough, we would have
		\begin{align*}
			V_0((1+\eps)\bar u)  &= V_0(\bar u) + \eps \langle V'_0(\bar u),\bar u\rangle + o(\eps) >V_0(\bar u)\\
			G_\a((1+\eps)\bar u)  &= G_\a(\bar u) + \eps \l  \langle V'_0(\bar u),\bar u\rangle + o(\eps) <G_\a(\bar u)\\
		\end{align*} 
	So we deduce that there exists $\bar v\in \tilde X$, with $\bar v \neq 0$, such that $ V_0(\bar v) > V_0(\bar u)=1$ and $G_\a(\bar v) < G_\a (\bar u)=\inf_{u \in H}G_\alpha(u)$. This immediately leads to a contradiction, since there exists $\s \in (0,1)$  such that $\tilde v = \bar v(\cdot/\s)\in H$ (see \eqref{eq:rescaling}, and
		\begin{align*}
			G_\a(\tilde v)&= \frac 12 \ird |\n \bar v|^2\, dx +\frac{\s ^2}2 \ird (1+|\s x|^\a)\bar v ^2\, dx + \s^2\ird W(\bar v)\, dx\\
			&<G_\a(\bar v)<\inf_{u \in H}G_\alpha(u).
		\end{align*}
	Taking $q= 4 \l$, we get a solution of \eqref{eq:e1}. 
	
\end{proof}

	\section{Existence of solution for \eqref{eq:e2}}\label{sec3}
	
	In this section, we use variational methods to find solutions of \eqref{eq:e2} (or \eqref{eq:e2e}, equivalently). Our final result is the following

	\begin{theorem}
		\label{teo2}
		If $\alpha > 0$ and either $2 < p < 3$ or $p \ge 4$, then \eqref{eq:e2} possesses infinitely many solutions.
	\end{theorem}
	
	 In order to get it, let us show that the functional $J_\a$ associated to \eqref{eq:e2e}, satisfies the following geometrical properties
	
	\begin{proposition}\label{pr:mp}
	Take $r>\frac{\a-2}{2}$, $w\in X\setminus\{0\}$ and for all $t>0$ set $w_t(\cdot)=t^r w(\cdot/t).$ Then there exist $\rho, \beta , \bar t> 0$  such that
	\begin{itemize}
	\item [$i)$] $J_\alpha(u) \geq \beta$ for all $u \in X$ such that $\|u\| = \rho$;
	\item [$ii)$] $J_\alpha(w_{\bar t}) < 0$ and $\|w_{\bar t}\| > \rho$.
	\end{itemize}
	\end{proposition}
	\begin{proof}
	Note that, for any $u \in X$, from \eqref{embeddingX} and \eqref{ineq:V1}
	\begin{eqnarray*}
	J_\alpha(u) & = & \frac{1}{2}\|u\|^2 - \frac{1}{4}V_1(u) + \frac{1}{4}V_2(u) - \frac{1}{p}\|u\|_p^p\\
	& \geq &  \frac{1}{2}\|u\|^2 - \frac{C_\alpha}{4\pi}\|u\|_2^2\|u\|_*^2 - \frac{C}{p}\|u\|_p^p\\
	& \geq &  \frac{1}{2}\|u\|^2 - \frac{C_\alpha}{4\pi}\|u\|^4 - \frac{C}{p}\|u\|_p^p\\
	& = & \|u\|^2\left( \frac{1}{2} - \frac{C_\alpha}{4\pi}\|u\|^2 - \frac{C}{p}\|u\|^{p-2}\right)\\
	& \geq & \rho^2\left( \frac{1}{2} - \frac{C_\alpha}{4\pi}\rho^2 - \frac{C}{p}\rho^{p-2}\right)\\
	& =: & \beta,
	\end{eqnarray*}
	where $\rho > 0$ is such that $\frac{1}{2} - \frac{C_\alpha}{4\pi}\rho^2 - \frac{C}{p}\rho^{p-2} > 0$. This proves $i)$.
	
	As to $ii)$, we compute
	\begin{eqnarray*}
	J_\alpha(w_t) & = & \frac{t^{2r}}{2}\int_{\mathbb{R}^2}|\nabla w|^2 dx + \frac{t^{2r + 2}}{2}\int_{\mathbb{R}^2}\left(1 + t^\alpha |x|^\alpha\right)w^2 dx \\
	& & - \frac{t^{4r+4} \mbox{log}\, t}{8\pi}\|w\|_2^4 - \frac{t^{4r+4}}{4}V_0(w) - \frac{t^{pr+2}}{p}\|w\|_p^p.\\
	\end{eqnarray*}
	Then, by the choice of $r $, it follows that $J_\alpha(w_t) \to -\infty$ as $t \to +\infty$. Moreover, since $\|w_t\| \to +\infty$ as $t \to +\infty$, we can choose $\bar t$ such that such that $J_\alpha(w_{\bar t}) < 0$ and $\|w_{\bar t}\| > \rho$, which proves $ii)$.
	
	\end{proof}

	Now take $k\in \N$, $k\ge 1$ and consider the following  constraint introduced by Du and Weth in \cite{DuWeth} (see $(iii)$ in Example 1.1.):
		\begin{equation*}
			X_k:=\{u\in X;{\bf A^j}\star u =u, \, \forall j=1, 2, \ldots, 2k\}
		\end{equation*}
	where the operator ${\bf A^j\star}:X\to X$ is such that 
	$$u\in X\mapsto (-1)^j(u\circ {\bf A^j})$$
	and ${\bf A^j}:\RD\to\RD$ is such that for any $x\in\RD:$ ${\bf A^j}(x)=A^j\cdot x$ with
	$$  A^j=
\begin{pmatrix}
	\cos  \frac{j\pi}k& -\sin \frac{j\pi}k\\ 
	\sin \frac{j\pi}k & \cos  \frac{j\pi}k  \\ 
\end{pmatrix}.
		$$
	Of course, apart from the null functions, all functions in $X_k$ are sign changing.  Moreover	by arguments based on Palais symmetric criticality principle, it can be proved that $X_k$ is a natural constraint for $J_\a$. In what follows we are interested in finding a critical point of $J_\a|_{X_k}$ at the mountain pass level.
	
	First observe that, since for every $w\in X_k\setminus\{0\}$ and $t>0$ we have $w_t\in X_k$, the set 
	$$
	\Gamma_k = \left\{ \gamma \in C([0,1],X_k); \, \gamma(0) = 0 \,\, \mbox{and} \,\, J_\a(\gamma(1))<0, \|\gamma(1)\|>\rho\right\}
	$$

	is nonempty by $ii)$ of Proposition \ref{pr:mp}. Then, by $i)$ of Proposition \ref{pr:mp}, the number
		$$
	c_{\alpha,k} = \inf_{\gamma \in \Gamma_k} \max_{0 \leq t \leq 1}J_\alpha(\gamma(t))
	$$
	is well defined and larger than $\beta$.

	\begin{proposition}
	\label{propMPT}
	There exists $(u_n)_{n \in \mathbb{N}} \subset X_k$ such that, as $n \to +\infty$,
	$$
	J_\alpha(u_n) \to c_{\alpha,k},
	$$
	$$
	J_\alpha'(u_n) \to 0
	$$
	and
	\begin{equation}\label{poho}
		P_\a(u_n) \to 0,
	\end{equation}
	where for an arbitrary  $r\in\R$, 
	\begin{eqnarray*}
	P_\a(u) & = & r\|\nabla u\|_2^2 + (r+1)\|u\|_2^2 + \left(r + 1+ \frac{\alpha}{2}\right) \int_{\mathbb{R}^2}|x|^\alpha u^2 dx\\
	& & - \frac{q}{8\pi}\|u\|_2^4 - q(r+1)V_0(u) - \frac{pr+2}{p}\|u\|_p^p.\nonumber
	\end{eqnarray*}
	\end{proposition}
	The proof follows by the same arguments as \cite[Lemma 3.5]{A}.
	
	Note that, by the last result, there exists $(u_n)_{n \in \mathbb{N}} \subset X_k$ a Palais-Smale sequence for $J_\alpha$ at the level $c_{\alpha,k}$. On the other hand, in the next result, we prove that such a sequence is bounded in $X_k$.

	\begin{proposition}
	If $p\ge 4$, then any Palais-Smale sequence for $J_\a$ is bounded.\\
	If $2 < p < 3$, then the sequence $(u_n)_{n \in \mathbb{N}}$ coming from Propositon \ref{propMPT} is bounded in $X_k$.
	\label{pro:bound}
	\end{proposition}
	\begin{proof}
	First of all, consider the case $p \ge 4$. Assume that $(v_n)_n$ in $X$ is a Palais-Smale sequence for $J_\a$. Then, since $J_\a'(v_n)\to 0$, we have
	
	\begin{align}
	\label{eqNehari}
	o_n(1)\|v_n\|&=\langle J'_\a(v_n),v_n\rangle\nonumber\\
	&=\|\nabla v_n\|_2^2 + \|v_n\|_2^2 + \int_{\mathbb{R}^N}|x|^\alpha v_n^2 dx - qV_0(v_n) - \|v_n\|_p^p \nonumber\\
	& :=N_\a(v_n).
	\end{align}
	Moreover, since $J_\a(v_n)$ is bounded, then there exists $M>0$ such that
	\begin{eqnarray*}
	M+o_n(1)\|v_n\| & \ge & J_\alpha(v_n) - \frac{1}{4}N_\a(v_n)\\
	& = & \frac{1}{4}\|\nabla v_n\|_2^2 + \frac{1}{4}\int_{\mathbb{R}^2}(1+|x|^\alpha) v_n^2 dx +\left( \frac{1}4-\frac 1p\right)\|v_n\|_p^p.
	\end{eqnarray*}
	This, in turn, implies that $(v_n)_n$ is bounded in $X_k$. 
	
	Now assume $2 < p < 3$ and consider $(u_n)_n$ as in Proposition \ref{propMPT}. Applying \eqref{poho} we obtain
	
	\begin{eqnarray}\label{eq:B1} 
		c_{\alpha,k} + o_n(1) & = & J_\alpha(u_n) - \frac{1}{4(r+1)}P_\alpha(u_n)\\ 
		\nonumber & = & \left(\frac{r+2}{4(r+1)}\right)\|\nabla u_n\|_2^2 + \frac{1}{4}\|u_n\|_2^2\\
		\nonumber & &  + \left(\frac{2(r+1)-\alpha}{8(r+1)}\right)\int_{\mathbb{R}^2}|x|^\alpha u_n^2 dx\\
		\nonumber&& + \left(\frac{(p-4)r-2}{4p(r+1)}\right)\|u_n\|_p^p + \frac{q}{32\pi(r+1)}\|u_n\|_2^4.
	\end{eqnarray}
	If we take $r > \frac{\alpha-2}{2}$, there exist $C_i > 0$, $i \in \{1,...,5\}$, such that
	\begin{equation}
	\label{eq:B2}
	c_{\alpha,k} + o_n(1) = C_1\| \nabla u_n\|_2^2 + C_2\|u_n\|_2^2 - C_3\|u_n\|_p^p + C_4\int_{\mathbb{R}^2}|x|^\alpha u_n^2 dx + C_5\|u_n\|_2^4.
	\end{equation}
	
	Suppose by contradiction that, up to a subsequence, $\|\nabla u_n\|_2 \to +\infty$, as $n \to +\infty$. Let $t_n = \|\nabla u_n\|_2^{-\frac{1}{2}}$ and note that $\lim_{n \to +\infty}	t_n = 0$. Define
	$$
	v_n(x)=t_n^2 u_n(t_n x),
	$$
	in such a way that, for all $1 \leq q < +\infty$,
	\begin{equation}
	\|\nabla v_n\|_2^2 = t_n^4\|\nabla u_n\|_2^2 = 1, \quad \|v_n\|_q^q = t_n^{2q-2}\|u_n\|_q^q.
	\label{eq:B3}
	\end{equation}

	Then, multiplying \eqref{eq:B2} by $t_n^4$, it follows that
	\begin{eqnarray}	\label{eq:B6}
	 c_{\alpha,k} t_n^4 + o(t_n^4) & = & C_1t_n^4 \|\nabla u_n\|_2^2 + C_2t_n^4\|u_n\|_2^2\\
	& & - C_3 t_n^4\|u_n\|_p^p + C_4 t_n^4\int_{\mathbb{R}^2}|x|^\alpha u_n^2 dx + C_5 t_n^4\|u_n\|_2^4.\nonumber
	\end{eqnarray}
	Moreover, by Gagliardo-Nirenberg's inequality,
	\begin{equation}
	\|u_n\|_p^p \leq C\|u_n\|_2^2 \|\nabla u_n\|_2^{p-2} = C t_n^{4-2p}\|u_n\|_2^2.
	\label{eq:B5}
	\end{equation}
	Then,
	\begin{equation}
	t_n^4\|u_n\|_p^p \leq C t_n^{6-2p}\|v_n\|_2^2.
	\label{eq:B7}
	\end{equation}
	By \eqref{eq:B6} and \eqref{eq:B7}, it follows that
	\begin{eqnarray*}
	\overline{C}t_n^4 & \geq & c_\alpha t_n^4 + o(t_n^4)\\
	& \geq & C_5\|v_n\|_2^4 - C t_n^{6-2p}\|v_n\|_2^2.
	\nonumber
	\end{eqnarray*}

	Hence, we can see that there exists $\tilde C > 0$ such that, for $n$ sufficiently large,
	\begin{equation}
	\label{eq:B9}
	{\|v_n\|_2 \leq \tilde Ct_n^{3-p}}.
	\end{equation}
	Moreover, from \eqref{eq:B6}, \eqref{eq:B7} and \eqref{eq:B9}, it follows that
	\begin{equation}
	\label{eq:B10}
	t_n^4\int_{\mathbb{R}^2}|x|^\alpha u_n^2 dx \leq c_{\alpha,k} t_n^4 + o(t_n^4) + C_3t_n^4\|u_n\|_p^p = o_n(1).
	\end{equation}
		
	Moreover, by \eqref{eq:B9},
	\begin{equation}
	t_n^4V_0(u_n) = V_0(v_n) + \|v_n\|_2^4 \log (t_n) = V_0(v_n) + o_n(1).
	\label{eq:B11}
	\end{equation}
	From Proposition \ref{propMPT} applying \eqref{poho} for $r=0$, it follows that
	\begin{eqnarray}	\label{eq:B12} 
	o(t_n^4) & = & t_n^4 P_\a(u_n) \\
	\nonumber & = & t_n^4\|u_n\|_2^2 + \left(\frac{2+\alpha}{2}\right)t_n^4\int_{\mathbb{R}^2}|x|^\alpha u_n^2 dx - q t_n^4 V_0(u_n)\\
	\nonumber & & - \frac{q}{8\pi}t_n^4 \|u_n\|_2^4 - \frac{2}{p}t_n^4\|u_n\|_p^p\\
	\nonumber & = & t_n^2 \|v_n\|_2^2 + \left(\frac{2+\alpha}{2}\right)t_n^4\int_{\mathbb{R}^2}|x|^\alpha u_n^2 dx - \frac{q}{8\pi}\|v_n\|_2^4\\
	 & & - \frac{2}{p}t_n^4\|u_n\|_p^p - qV_0(v_n) + o_n(1). \nonumber
	\end{eqnarray}
	From \eqref{eq:B7} and \eqref{eq:B9}, it follows that
	\begin{equation}
	t_n^4\|u_n\|_p^p \leq C\|v_n\|_2^2 t_n^{6-2p} = o_n(1).
	\label{eq:B13}
	\end{equation}
	Hence, \eqref{eq:B10}, \eqref{eq:B12} and \eqref{eq:B13} imply that
	\begin{equation}
	V_0(v_n) = o_n(1).
	\label{eq:B15}
	\end{equation}
	
	Then, multiplying \eqref{eqNehari} by $t_n^4$ and using \eqref{eq:B3}, \eqref{eq:B10}, \eqref{eq:B11}, \eqref{eq:B13} and \eqref{eq:B15}, we have that
	\begin{eqnarray*}
	o(t_n^4)\|u_n\| & = & t_n^4\|\nabla u_n\|_2^2 + t_n^4\|u_n\|_2^2 + t_n^4\int_{\mathbb{R}^2}|x|^\alpha u_n^2 dx\\
	& & -qt_n^4V_0(u_n) - t_n^4\|u_n\|_p^p\\
	& = & 1 + o_n(1),
	\end{eqnarray*}
	which is a contradiction since, by \eqref{eq:B3} and \eqref{eq:B10},
		\begin{align*}
			t_n^4\|u_n\| = \left(t_n^4+ t_n^8\int_{\mathbb{R}^2}|x|^\alpha u_n^2 dx\right)^{\frac 12}=o_n(1).
		\end{align*}
	Hence, $(\nabla u_n)_n$ is bounded in $L^2(\mathbb{R}^2)$.
	
	Moreover, by \eqref{eq:B2} and \eqref{eq:B5} and taking into account that there exists $C_6 > 0$ such that $\|\nabla u_n\|_2^2 \leq C_6$ for all $n \in \mathbb{N}$, we have
	$$
	C_5\|u_n\|_2^4 + (C_2-C C_3C_6^{p-2})\|u_n\|_2^2 + C_4\int_{\mathbb{R}^2}|x|^\alpha u_n^2 dx \leq c_{\alpha,k} + o_n(1),
	$$
	The last estimate, in turn, implies that $(u_n)_n$ is bounded in $L^2(\mathbb{R}^N)$ and also in $X_k$.
	
	\end{proof}

	\begin{proof}[Proof of Theorem \ref{teo2}]
	Take $(u_n)_n$ as in Proposition \ref{pro:bound}. It follows that there exists $u_k \in X_k$, such that
	$$
	u_n \rightharpoonup u_k, \quad \mbox{in $X_k$}.
	$$
	From the compact embeddings $X_k \hookrightarrow L^r(\mathbb{R}^2)$ for $r \geq 2$, we have that
	\begin{equation}
	u_n \to u_k, \quad \mbox{in $L^r(\mathbb{R}^2)$, for $r \geq 2$}.
	\label{eq:unuLr}
	\end{equation}
	Hence, the same arguments as in Lemma \ref{le:weaklycontinuous} and \eqref{eq:unuLr} imply that $u_k$ is a weak solution to \eqref{eq:e2e}. Moreover, since $\langle J_\alpha'(u_n),u_n\rangle =o_n(1)$ and $\langle J_\alpha'(u),u\rangle = 0$, by Lemma \ref{le:weaklycontinuous} it follows that
	$$
	\|u_n\|^2 = \frac q4 \langle V_0'(u_n), u_n \rangle+ \|u_n\|_p^p + o_n(1) = \frac q4\langle V_0'(u_k), u_k\rangle + \|u_k\|_p^p + o_n(1) = \|u_k\|^2 + o_n(1).
	$$
	Hence, from the last equality it follows that
	$$
	u_n \to u_k \quad \mbox{in $X$}
	$$
	and then, 
	$$
	J_\alpha(u_k) = c_{\alpha,k}.
	$$
	Since $c_{\alpha,k} > 0$, we have that $u_k\in X_k$ is a nontrivial sign-changing solution of \eqref{eq:e2e} at the mountain pass minimax level.\\
	Now, in order to prove multiplicity, we can proceed exactly as in \cite[Proof of Corollary 1.4.]{DuWeth} to obtain a sequence of solutions $(u_{3^h})_h$ such that, for all $h\ge 1$
		\begin{itemize}
			\item $u_{3^h}\in X_{3^h}$,
			\item $J_\a(u_{3^h})=c_{\a,3^h}$,
		\end{itemize}
	being the sequence $(c_{\a,3^h})_h$ of the mountain pass levels of  $J_\a$ on $X_{3^h}$ non-decreasing and unbounded from above.
	\end{proof}
	{\bf Acknowledgments:} Antonio Azzollini is supported by PRIN 2017JPCAPN {\em Qualitative and quantitative aspects of nonlinear PDEs}. \\
	Marcos T. O. Pimenta is partially supported by FAPESP 2021/04158-4 and CNPq 303788/2018-6.

\end{document}